\documentclass[oneside,english]{amsart}
\usepackage[T1]{fontenc}
\usepackage[latin9]{inputenc}
\usepackage{geometry}
\usepackage{amstext}
\usepackage{amsthm}
\usepackage{amssymb}
\usepackage{mathdots}
\usepackage{wasysym}

\makeatletter
\numberwithin{equation}{section}
\numberwithin{figure}{section}
\theoremstyle{plain}
\newtheorem{thm}{\protect\theoremname}[section]
  \theoremstyle{plain}
  \newtheorem{prop}[thm]{\protect\propositionname}
  \theoremstyle{remark}
  \newtheorem{rem}[thm]{\protect\remarkname}
  \theoremstyle{plain}
  \newtheorem{cor}[thm]{\protect\corollaryname}
  \theoremstyle{definition}
  \newtheorem{example}[thm]{\protect\examplename}

\makeatother

\usepackage{babel}
  \providecommand{\corollaryname}{Corollary}
  \providecommand{\examplename}{Example}
  \providecommand{\propositionname}{Proposition}
  \providecommand{\remarkname}{Remark}
\providecommand{\theoremname}{Theorem}

\begin{document}

\title{On Binomial Identities in Arbitrary Bases}

\author{Lin Jiu and Christophe Vignat}

\address{Department of Mathematics, Tulane University, New Orleans, USA and
L.s.s. Supelec, Universite Paris-Sud, Orsay, France}

\email{ljiu@tulane.edu, cvignat@tulane.edu}
\begin{abstract}
We extend the digital binomial identity as given by Nguyen el al.
to an identity in an arbitrary base $b$, by introducing the $b-$ary
binomial coefficients. We then study the properties of these coefficients
such as orthogonality, a link to Lucas' theorem and the corresponding
$b-$ary Pascal triangles. 
\end{abstract}


\maketitle

\section{Introduction}

In a series of recent articles, H.D. Nguyen \cite{Nguyen Generalization,Nguyen Digital}
himself and also together with T. Mansour \cite{Sheffer,Nguyen q-digital},
have introduced different versions of the binomial identity, in which
the usual integer powers are replaced by arithmetical functions that
count the digits of these powers in some base $b.$ See for example
\cite[Section 3]{Allouch-Shallit}.

Choose an integer $n$ in base $b$, then denote $S_{b}^{\left(k\right)}\left(n\right)$
the number of $k'$s in the $b-$ary expansion
\begin{equation}
n=\sum_{i}n_{i}b^{i}\label{eq:n base b}
\end{equation}
and $S_{b}\left(n\right)$ the sum of these digits, 
\begin{equation}
S_{b}\left(n\right)=\sum_{j=0}^{b-1}kS_{b}^{\left(k\right)}\left(n\right).\label{eq:Sb(n)}
\end{equation}
The main extension of the binomial identity as given by Nguyen is
for base $b=2$ and reads as follows \cite{Nguyen Generalization,Nguyen Digital}

\[
\left(X+Y\right)^{S_{2}\left(n\right)}=\sum_{\underset{\left(k,n-k\right)\text{carry-free}}{0\le k\le n}}X^{S_{2}\left(k\right)}Y^{S_{2}\left(n-k\right)},
\]
where the sum is over all values of $k$ such that the addition of
$k$ and $n-k$ in base $b$ is carry-free; remarking that this condition
is verified if and only if $S_{2}\left(k\right)+S_{2}\left(n-k\right)=S_{2}\left(n\right)$,
Nguyen's formula is also stated equivalently in the beautifully symmetric
form 
\[
\left(X+Y\right)^{S_{2}\left(n\right)}=\sum_{S_{2}\left(k\right)+S_{2}\left(n-k\right)=S_{2}\left(n\right)}X^{S_{2}\left(k\right)}Y^{S_{2}\left(n-k\right)}.
\]
This result is proven in \cite{Nguyen Digital} using a polynomial
generalization of the Sierpinski triangle, which is the Pascal triangle
modulo $2$.

An extension of this result to an arbitrary base $b$ is given in
\cite{Nguyen Generalization} in the form
\begin{equation}
\prod_{i=0}^{N-1}\binom{x+y+n_{i}-1}{n_{i}}=\sum_{0\le k\apprle_{b}n}\prod_{i=0}^{N-1}\binom{x+k_{i}-1}{k_{i}}\binom{x+n_{i}-k_{i}-1}{n_{i}-k_{i}},\label{eq:product}
\end{equation}
where the summation range $0\le k\apprle_{b}$ is over all integers
$k$ such that the digits of which satisfy $k_{i}\le n_{i},\thinspace\thinspace\text{for all}\thinspace\thinspace i\in\left\{ 0,\dots,N-1\right\} .$

The aim of this paper is to show that these results are a consequence
of an elementary identity stated in the next section. This approach
makes Nguyen's results more accessible and provides a number of generalizations.
The paper is organized as follows.

In Section 2, we state a general formula on the sum over digits of
an integer number. In Section 3, we provide a new version of the binomial
expansion in base $b$ that involves a corresponding $b-$ary binomials
coefficients; some of the properties of these coefficients such as
the link with Kummer's theorem and orthogonality are presented in
Section 4. The last section exhibits the construction rules for a
Pascal type triangle built from these coefficients.

\section{A general formula }

Formula (\ref{eq:product}) is proved by Nguyen using a polynomial
extension of Sierpinski's matrices. We show here another more elementary
approach. 

We start to remark that the formula (\ref{eq:product}) can be restated
equivalently as 
\[
\prod_{i=0}^{N-1}\frac{\left(x+y\right)_{n_{i}}}{n_{i}!}=\prod_{i=0}^{N-1}\sum_{0\le k_{i}\le n_{i}}\frac{\left(x\right)_{k_{i}}}{k_{i}!}\frac{\left(y\right)_{n_{i}-k_{i}}}{\left(n_{i}-k_{i}\right)!},
\]
where $\left(x\right)_{k}=\frac{\Gamma\left(x+k\right)}{\Gamma\left(x\right)}$
is the Pochhammer symbol. Next we realize that this identity holds
in fact component-wise, i.e. for all $i\in\left\{ 0,\dots,N-1\right\} $
and arbitrary integer $n_{i},$ it holds that
\begin{equation}
\frac{\left(x+y\right)_{n_{i}}}{n_{i}!}=\sum_{k_{i}=0}^{n_{i}}\frac{\left(x\right)_{k_{i}}}{k_{i}!}\frac{\left(y\right)_{n_{i}-k_{i}}}{\left(n_{i}-k_{i}\right)!}\label{eq:ChuVandermonde}
\end{equation}
which is actually the Chu-Vandermonde identity. Since this identity
is a consequence of the fact that the Pochhammer sequence $\left(x\right)_{k}$
is a binomial-type sequence\footnote{a sequence of polynomials $p_{n}\left(x\right)$ is of binomial type
(see \cite[p.26]{Roman}) if it satisfies the convolution identity
\[
p_{n}\left(x+y\right)=\sum_{k=0}^{n}\binom{n}{k}p_{k}\left(x\right)p_{n-k}\left(y\right).
\]
}, it suggests the following result.
\begin{prop}
Assume that $\left\{ a_{n}\right\} ,\thinspace\thinspace\left\{ b_{n}\right\} $
and $\left\{ c_{n}\right\} $ are three sequences related as
\begin{equation}
c_{n}=\sum_{k=0}^{n}\binom{n}{k}a_{k}b_{n-k},\label{eq:Convolution}
\end{equation}
then
\begin{equation}
\prod_{i=0}^{N-1}\frac{c_{n_{i}}}{n_{i}!}=\sum_{0\le k\apprle_{b}n}\prod_{i=0}^{N-1}\frac{a_{k_{i}}}{k_{i}!}\frac{b_{n_{i}-k_{i}}}{\left(n_{i}-k_{i}\right)!}.\label{eq:Product_Convolution}
\end{equation}

\end{prop}
We now apply this general formula to obtain a base $b$ generalization
of the binomial identity.

\section{A generalized Binomial Identity}
\begin{thm}
With the notations (\ref{eq:n base b}) and (\ref{eq:Sb(n)}), the
identity 
\begin{equation}
\left(X+Y\right)^{S_{b}\left(n\right)}=\sum_{k=0}^{n}{n \choose k}_{b}X^{S_{b}\left(k\right)}Y^{S_{b}\left(n-k\right)}.\label{eq:binomial}
\end{equation}
holds for all $X,Y\in\mathbb{C},$ where the $b-$ary binomial coefficients
$\binom{n}{k}_{b}$ are defined as follows. Suppose $n$ and $k$
have expansions 
\[
\begin{cases}
n=\underset{l=0}{\overset{N_{n}-1}{\sum}}n_{l}b^{l} & ,\\
k=\underset{l=0}{\overset{N_{k}-1}{\sum}}k_{l}b^{l} & ,
\end{cases}
\]
in base $b$, where $N_{n}$ and $N_{k}$ are the number of digits
of $n$ and $k$, respectively in base $b$. Then letting $N:=\max\left\{ N_{n},N_{k}\right\} $,
we have 
\begin{equation}
\binom{n}{k}_{b}=\overset{N-1}{\underset{l=0}{\prod}}{n_{l} \choose k_{l}}.\label{eq:binomial coefficients}
\end{equation}
\end{thm}
\begin{proof}
It only suffices to assume that $\forall l=0$, $\cdots$, $N-1$,
$0\le k_{l}\le n_{l}$, since other cases gives ${n \choose k}_{b}=0.$
In this case, we could compute 
\[
n-k=\sum_{l=0}^{N-1}\left(n_{l}-k_{l}\right)b^{l},
\]
which is equivalent to both 
\[
S_{b}\left(k\right)+S_{b}\left(n-k\right)=S_{b}\left(n\right)
\]
and also the assumption that the addition of $k$ and $n-k$ is carry-free
in base $b$, as mentioned before.

Applying the convolution relation (\ref{eq:Product_Convolution})
to get 
\begin{equation}
\prod_{l=0}^{N-1}c_{n_{l}}=\prod_{l=0}^{N-1}n_{l}!\sum_{k_{l}=0}^{n_{l}}\frac{a_{k_{l}}}{j_{k}!}\cdot\frac{b_{n_{l}-k_{l}}}{\left(n_{l}-k_{l}\right)!}=\sum_{S_{b}\left(k\right)+S_{b}\left(n-k\right)=S_{b}\left(n\right)}\left(\prod_{l=0}^{N-1}{n_{l} \choose k_{l}}a_{k_{l}}b_{n_{l}-k_{l}}\right).\label{eq:Digits_Convolution}
\end{equation}
Now the choice 
\[
\begin{cases}
c_{n_{l}}:=\left(X+Y\right)^{n_{l}} & ,\\
a_{k_{l}}:=X^{k_{l}} & ,\\
b_{n_{l}-k_{l}}:=Y^{n_{l}-k_{l}} & ,
\end{cases}
\]
satisfies the property (\ref{eq:Convolution}); with this choice,
we obtain 
\[
\prod_{l=0}^{N-1}\left(X+Y\right)^{n_{l}}=\sum_{S_{b}\left(k\right)+S_{b}\left(n-k\right)=S_{b}\left(n\right)}\left(\prod_{l=0}^{N-1}{n_{l} \choose k_{l}}X^{k_{l}}Y^{n_{l}-k_{l}}\right).
\]
Notice that the left-hand side reads 
\[
\prod_{l=0}^{N-1}\left(X+Y\right)^{n_{l}}=\left(X+Y\right)^{\overset{N-1}{\underset{l=0}{\sum}}n_{l}}=\left(X+Y\right)^{S_{b}\left(n\right)},
\]
while the right-hand side is 
\[
\prod_{l=0}^{N-1}\sum_{k_{l}=0}^{n_{l}}{n_{l} \choose k_{l}}X^{k_{l}}Y^{n_{l}-k_{l}}=\sum_{k=0}^{n}\left(\prod_{l=0}^{N-1}{n_{l} \choose k_{l}}\right)X^{\overset{N-1}{\underset{l=0}{\sum}}k_{l}}Y^{\overset{N-1}{\underset{l=0}{\sum}}\left(n_{l}-k_{l}\right)}=\sum_{k=0}^{n}\left(\prod_{l=0}^{N-1}{n_{l} \choose k_{l}}\right)X^{S_{b}\left(k\right)}Y^{S_{b}\left(n-k\right)}.
\]
This completes the proof.\end{proof}
\begin{rem}
\label{rem:remark 3.2}When $b>n$, then in the $b$-ary expansion,
$n$ has only one digit, i.e., $n=n_{0}$. And so do $k$ and $n-k$,
since $k\le n$. Then, we have 
\[
\begin{cases}
S_{b}\left(n\right)=n_{0}=n & ,\\
S_{b}\left(k\right)=k,\ S_{b}\left(n-k\right)=n-b & ,\\
{n \choose k}_{b}={n \choose k} & ,
\end{cases}
\]
so that \ref{eq:binomial} reduces to 
\[
\left(X+Y\right)^{n}=\sum_{k=0}^{n}{n \choose k}X^{k}Y^{n-k},
\]
namely the usual binomial identity.\end{rem}
\begin{cor}
When $b=2$, i.e. in the binary case, we recover the identity 
\[
\left(X+Y\right)^{S_{2}\left(n\right)}=\sum_{S_{2}\left(k\right)+S_{2}\left(n-k\right)=S_{2}\left(n\right)}X^{S_{2}\left(k\right)}Y^{S_{2}\left(n-k\right)}
\]
since the coefficients $\binom{n}{k}_{2}$ take the value $\binom{0}{0}=\binom{1}{0}=\binom{1}{1}=1$
and $0$ otherwise.

In the case $b=3$, all $\binom{n}{k}_{3}=0$ if $S_{3}\left(k\right)+S_{3}\left(n-k\right)\ne S_{3}\left(n\right);$
otherwise, $\binom{n}{k}=1$ except for $\binom{2}{1}=2$. Therefore,
we have the equivalent expression 
\[
\left(X+Y\right)^{S_{3}\left(n\right)}=\sum_{S_{3}\left(k\right)+S_{3}\left(n-k\right)=S_{3}\left(n\right)}2^{S_{3}^{\left(2\right)}\left(n\right)-S_{3}^{\left(2\right)}\left(k\right)-S_{3}^{\left(2\right)}\left(n-k\right)}X^{S_{3}\left(k\right)}Y^{S_{3}\left(n-k\right)}.
\]
\end{cor}
\begin{rem}
Different choices of $\left\{ a_{n}\right\} $, $\left\{ b_{n}\right\} $
and $\left\{ c_{n}\right\} $ lead to different identities. For instance,
another possible choice is 
\[
\begin{cases}
c_{n}=\left(X+Y\right)_{n} & ,\\
a_{n}=\left(X\right)_{n} & ,\\
b_{n}=\left(Y\right)_{n} & .
\end{cases}
\]
The generating function
\[
\sum_{n\ge0}\frac{\left(X\right)_{n}}{n!}z^{n}=\frac{1}{\left(1-z\right)^{X}}
\]
shows that the convolution property (\ref{eq:Convolution}) holds.
Then from (\ref{eq:Digits_Convolution}), we deduce 
\[
\prod_{l=0}^{N-1}\frac{\left(X+Y\right)_{n_{l}}}{n_{l}!}=\sum_{S_{b}\left(k\right)+S_{b}\left(n-k\right)=S_{b}\left(n\right)}\left(\prod_{l=0}^{N-1}\frac{X_{j_{l}}}{j_{l}!}\cdot\frac{Y_{n_{l}-j_{l}}}{\left(n_{l}-j_{l}\right)!}\right),
\]
or equivalently 
\[
\prod_{l=0}^{N-1}{X+Y+n_{l}-1 \choose n_{l}}=\sum_{0\le k\apprle_{b}n}\left[\prod_{l=0}^{N-1}{X+j_{l}-1 \choose j_{l}}{Y+n_{l}-j_{l}-1 \choose n_{l}-j_{l}}\right],
\]
which appears as \cite[Thm.2]{Nguyen Generalization}. Then, for the
binary case $b=2$, $n_{l}\in\left\{ 0,1\right\} $ yields $n_{l}!=1$
for all $l$, and from 
\[
\left(X+Y\right)_{n_{l}}=\begin{cases}
X+Y, & \text{ if }n_{l}=1\\
0, & \text{ if }n_{l}=0
\end{cases}
\]
we deduce 
\[
\left(X+Y\right)^{S_{2}\left(n\right)}=\sum_{S_{2}\left(k\right)+S_{2}\left(n-k\right)=S_{2}\left(n\right)}X^{S_{2}\left(k\right)}Y^{S_{2}\left(n-k\right)}.
\]

\end{rem}

\section{Properties of the Binomial Coefficients}

We state in this sections some properties of the generalized binomial
coefficients defined by (\ref{eq:binomial coefficients}). The first
one is a generating function for these coefficients.

\subsection{Generating function}
\begin{thm}
A generating function for the $b-$ary binomial coefficients $\binom{n}{k}_{b}$
is 
\begin{equation}
\sum_{k=0}^{n}{n \choose k}_{b}x^{k}=\prod_{l=0}^{N}\left(1+x^{b^{l}}\right)^{n_{l}}.\label{eq:b-ary binomial generating function}
\end{equation}
\end{thm}
\begin{proof}
Define the right hand side as $P\left(x\right)$, a polynomial with
degree
\[
\deg P=\overset{N-1}{\underset{l=0}{\sum}}n_{l}b^{l}=n.
\]
Denote $a_{k}^{\left(n\right)}$ the coefficient of $x^{k}$ in $P\left(x\right)$
and remark that $a_{k}^{\left(n\right)}=0$ if and only if there is
a carry in the addition of $k$ and $n-k$ in base $b.$ An elementary
enumeration shows that 
\[
a_{k}^{\left(n\right)}={n_{N-1} \choose k_{N-1}}{n_{N-2} \choose k_{N-2}}\cdots{n_{0} \choose k_{0}}
\]
which gives the desired result.
\end{proof}

\subsection{Multinomial version}

The next property is the extension of the previous result to the multinomial
case.
\begin{thm}
{[}The Multinomial Version{]} Define the multinomial coefficient 
\[
{n \choose k_{1},\cdots,k_{m}}_{b}=\prod_{l=0}^{N-1}{n_{l} \choose \left(k_{1}\right)_{l},\cdots,\left(k_{m}\right)_{l}}_{b}
\]
where $\left(k_{i}\right)_{l}$ denotes the rank-$l$ digit in the
expression of $k_{i}$ in base $b.$ 

Then 
\[
\left(X_{1}+\cdots+X_{m}\right)^{S_{b}\left(n\right)}=\sum_{k_{1},\cdots,k_{m}}{n \choose k_{1},\cdots,k_{m}}_{b}X_{1}^{S_{b}\left(k_{1}\right)}\cdots X_{m}^{S_{b}\left(k_{m}\right)}.
\]
\end{thm}
\begin{proof}
The proof is straightforward from the binomial expansion (\ref{eq:binomial})
and (\ref{eq:binomial coefficients}).
\end{proof}

\subsection{Symmetries}
\begin{thm}
The $b-$ary binomial coefficients satisfy\end{thm}
\begin{enumerate}
\item the symmetry property
\[
\binom{n}{k}_{b}=\binom{n}{n-k}_{b}
\]

\item the recurrence 
\[
\binom{n}{k}_{b}={n-b \choose k-b}_{b}+{n-b \choose k}_{b}.
\]
\end{enumerate}
\begin{proof}
The symmetry property is easily deduced from the definition (\ref{eq:binomial coefficients})
and the invariance $k_{l}\mapsto n_{l}-k_{l}$ of each term.

For the recurrence property, assume that each $k$ and $n$ have a
non-zero rank first digit, i.e. $k_{0}>0,\ n_{0}>0$. Then, $\left(n-b\right)_{0}=n_{0}-1$
and $\left(k-b\right)_{0}=k_{0}-1$ and
\[
\binom{n-b}{k-b}_{b}=\binom{n_{N-1}}{k_{N-1}}\cdots\binom{n_{0}-1}{k_{0}-1},\thinspace\thinspace\binom{n-b}{k}_{b}=\binom{n_{N-1}}{k_{N-1}}\cdots\binom{n_{0}-1}{k_{0}}
\]
and we deduce
\[
\binom{n-b}{k-b}_{b}+\binom{n-b}{k}_{b}=\binom{n_{N-1}}{k_{N-1}}\dots\left[\binom{n_{0}-1}{k_{0}-1}+\binom{n_{0}-1}{k_{0}}\right]=\binom{n_{N-1}}{k_{N-1}}\cdots\binom{n_{0}}{k_{0}}=\binom{n}{k}_{b}.
\]

This extends easily to the case where one or two of the numbers $\left(k_{0},n_{0}\right)$
is equal to $0$.
\end{proof}

\subsection{Link with Lucas' theorem}

The definition (\ref{eq:binomial coefficients}) will look familiar
to those readers who have already met Lucas' famous theorem which
we restate here (see also \cite{Fine_Binomial Modulo Prime p})
\begin{thm}
{[}Lucas{]} For $p$ a prime number, the binomial coefficients satisfy
\[
\binom{n}{k}\equiv\binom{n_{N}}{k_{N}}\dots\binom{n_{0}}{k_{0}}\mod p.
\]

\end{thm}
Under the same condition, this theorem is concisely rephrased in our
notations as
\[
\binom{n}{k}\equiv\binom{n}{k}_{p}\mod p.
\]

\begin{rem}
The Sierpinski matrix that is used by Nguyen et al. to prove their
results contains the coefficients $\binom{n}{k}\mod p$, or equivalently
$\binom{n}{k}_{p}\mod p$, which do not coincide with the coefficients
$\binom{n}{k}_{p}$ studied here, but are congruent to them.
\end{rem}

\subsection{Orthogonality Relations}

There are many elementary identities involving the usual binomial
coefficients that can be transferred to the case of the $b-$ary binomial
coefficients. Here, we only show one as an example. 
\begin{example}
If two sequences $\left\{ a_{n}\right\} $ and $\left\{ c_{n}\right\} $
are related as 
\[
a_{n}=\overset{n}{\underset{k=0}{\sum}}\left(-1\right)^{k}{n \choose k}c_{n}
\]
then 
\[
c_{n}=\overset{n}{\underset{k=0}{\sum}}\left(-1\right)^{k}{n \choose k}a_{k};
\]
these identities are know as inverse relations. Note that they are
equivalent to the orthogonality conditions 
\[
\sum_{k=j}^{n}{n \choose k}{k \choose j}\left(-1\right)^{k+j}=\delta_{n,j}=\begin{cases}
1 & j=n\\
0 & \text{ otherwise}
\end{cases}.
\]
The generalization to the $b-$ary case is as follows. \end{example}
\begin{thm}
The $b-$ary binomial coefficients satisfy the orthogonality relations
\begin{equation}
\sum_{k=j}^{n}{n \choose k}_{b}{k \choose j}_{b}\left(-1\right)^{S_{b}\left(k\right)+S_{b}\left(j\right)}=\delta_{n,j},\label{eq:b-binomial-inverse-relation}
\end{equation}
namely
\[
a_{S_{b}\left(n\right)}=\overset{n}{\underset{k=0}{\sum}}\left(-1\right)^{S_{b}\left(k\right)}{n \choose k}_{b}c_{S_{b}\left(n\right)}\Rightarrow c_{S_{b}\left(n\right)}=\overset{n}{\underset{k=0}{\sum}}\left(-1\right)^{S_{b}\left(k\right)}{n \choose k}_{b}a_{S_{b}\left(n\right)}.
\]
\end{thm}
\begin{proof}
Since
\[
{n \choose k}_{b}=\prod_{l=1}^{N-1}{n_{j} \choose k_{l}},\ \ {k \choose j}_{b}=\prod_{l=1}^{N-1}{k_{l} \choose j_{l}},\ \ S_{b}\left(k\right)=\sum_{l=0}^{N-1}k_{l}\text{ and }S_{b}\left(j\right)=\sum_{l=0}^{N-1}j_{l},
\]
we get
\[
\sum_{k=j}^{n}{n \choose k}_{b}{k \choose j}_{b}\left(-1\right)^{S_{b}\left(k\right)+S_{b}\left(j\right)}=\prod_{l=1}^{N-1}\sum_{k_{l}=j_{l}}^{n_{l}}{n_{j} \choose k_{l}}{k_{l} \choose j_{l}}\left(-1\right)^{k_{l}+j_{l}}=\prod_{l=1}^{N-1}\delta_{n_{l},j_{l}}=\delta_{n,j}.
\]

\end{proof}

\section{Pascal triangles}

In this last section, we look at the equivalent of Pascal triangles
that can be built from the $b-$ary binomial coefficients $\binom{n}{k}_{b}.$
Let us start with two examples, where we systematically replace each
null binomial entry with a dot (.) symbol to make the structure of
the triangle more apparent; remember that these null entries correspond
to the couples $\left(n,k\right)$ for which the addition of $k$
and $n-k$ is not carry-free in base $b.$
\begin{example}
In base\textbf{ $b=3,$ }the first 9 rows of the Pascal triangle are\textbf{
}
\[
T_{2}^{\left(3\right)}=\begin{array}{ccccccccccccccccc}
 &  &  &  &  &  &  &  & 1\\
 &  &  &  &  &  &  & 1 &  & 1\\
 &  &  &  &  &  & 1 &  & 2 &  & 1\\
 &  &  &  &  & 1 &  & . &  & . &  & 1\\
 &  &  &  & 1 &  & 1 &  & . &  & 1 &  & 1\\
 &  &  & 1 &  & 2 &  & 1 &  & 1 &  & 2 &  & 1\\
 &  & 1 &  & . &  & . &  & 2 &  & . &  & . &  & 1\\
 & 1 &  & 1 &  & . &  & 2 &  & 2 &  & . &  & 1 &  & 1\\
1 &  & 2 &  & 1 &  & 2 &  & 4 &  & 2 &  & 1 &  & 2 &  & 1
\end{array}
\]
Denote as 
\[
T_{1}^{\left(3\right)}=\begin{array}{ccccc}
 &  & 1\\
 & 1 &  & 1\\
1 &  & 2 &  & 1
\end{array}
\]
the first 3 top rows of this triangle and 
\[
*=\begin{array}{ccc}
. &  & .\\
 & .
\end{array}=\begin{array}{ccc}
0 &  & 0\\
 & 0
\end{array}
\]
the elementary reverse triangle built with zero entries, and notice
that
\[
T_{2}^{\left(3\right)}=\begin{array}{ccccc}
 &  & T_{1}^{\left(3\right)}\\
 & T_{1}^{\left(3\right)} & * & T_{1}^{\left(3\right)}\\
T_{1}^{\left(3\right)} & * & 2T_{1}^{\left(3\right)} & * & T_{1}^{\left(3\right)}
\end{array}.
\]
It appears that, with obvious notations,
\[
T_{2}^{\left(3\right)}=T_{1}^{\left(3\right)}\otimes T_{1}^{\left(3\right)}
\]
from which we deduce
\[
T_{m}^{\left(3\right)}=\left[T_{1}^{\left(3\right)}\right]^{\otimes m}.
\]
The operator $\otimes$ will be defined and discussed in the Proposition
\ref{prop:The-structure-of} below.
\end{example}

\begin{example}
In base 4, the first 16 rows of the binomial triangle $T_{4}^{\left(2\right)}$
are as follows.
\end{example}
{\tiny{}
\[
T_{4}^{\left(2\right)}=\begin{array}{ccccccccccccccccccccccccccccccc}
 &  &  &  &  &  &  &  &  &  &  &  &  &  &  & 1\\
 &  &  &  &  &  &  &  &  &  &  &  &  &  & 1 &  & 1\\
 &  &  &  &  &  &  &  &  &  &  &  &  & 1 &  & 2 &  & 1\\
 &  &  &  &  &  &  &  &  &  &  &  & 1 &  & 3 &  & 3 &  & 1\\
 &  &  &  &  &  &  &  &  &  &  & 1 &  & . &  & . &  & . &  & 1\\
 &  &  &  &  &  &  &  &  &  & 1 &  & 1 &  & . &  & . &  & 1 &  & 1\\
 &  &  &  &  &  &  &  &  & 1 &  & 2 &  & 1 &  & . &  & 1 &  & 2 &  & 1\\
 &  &  &  &  &  &  &  & 1 &  & 3 &  & 3 &  & 1 &  & 1 &  & 3 &  & 3 &  & 1\\
 &  &  &  &  &  &  & 1 &  & . &  & . &  & . &  & 2 &  & . &  & . &  & . &  & 1\\
 &  &  &  &  &  & 1 &  & 1 &  & . &  & . &  & 2 &  & 2 &  & . &  & . &  & 1 &  & 1\\
 &  &  &  &  & 1 &  & 2 &  & 1 &  & . &  & 2 &  & 4 &  & 2 &  & . &  & 1 &  & 2 &  & 1\\
 &  &  &  & 1 &  & 3 &  & 3 &  & 1 &  & 2 &  & 6 &  & 6 &  & 2 &  & 1 &  & 3 &  & 3 &  & 1\\
 &  &  & 1 &  & . &  & . &  & . &  & 3 &  & . &  & . &  & . &  & 3 &  & . &  & . &  & . &  & 1\\
 &  & 1 &  & 1 &  & . &  & . &  & 3 &  & 3 &  & . &  & . &  & 3 &  & 3 &  & . &  & . &  & 1 &  & 1\\
 & 1 &  & 2 &  & 1 &  & . &  & 3 &  & 6 &  & 3 &  & . &  & 3 &  & 6 &  & 3 &  & . &  & 1 &  & 2 &  & 1\\
1 &  & 3 &  & 3 &  & 1 &  & 3 &  & 9 &  & 9 &  & 3 &  & 3 &  & 9 &  & 9 &  & 3 &  & 1 &  & 3 &  & 3 &  & 1
\end{array}.
\]
}{\tiny \par}

Starting from the elementary triangles,
\[
T_{1}^{\left(4\right)}=\begin{array}{ccccccc}
 &  &  & 1\\
 &  & 1 &  & 1\\
 & 1 &  & 2 &  & 1\\
1 &  & 3 &  & 3 &  & 1
\end{array},\thinspace\thinspace*=\begin{array}{ccccc}
. &  & . &  & .\\
 & . &  & .\\
 &  & .
\end{array}=\begin{array}{ccccc}
0 &  & 0 &  & 0\\
 & 0 &  & 0\\
 &  & 0
\end{array},
\]
we notice that
\[
T_{2}^{\left(4\right)}=\begin{array}{ccccccc}
 &  &  & T_{1}^{\left(4\right)}\\
 &  & T_{1}^{\left(4\right)} & * & T_{1}^{\left(4\right)}\\
 & T_{1}^{\left(4\right)} & * & 2T_{1}^{\left(4\right)} & * & T_{1}^{\left(4\right)}\\
T_{1}^{\left(4\right)} & * & 3T_{1}^{\left(4\right)} & * & 3T_{1}^{\left(4\right)} & * & T_{1}^{\left(4\right)}
\end{array}
\]
and more generally
\[
T_{n}^{\left(4\right)}=\left[T_{1}^{\left(4\right)}\right]^{\otimes n}.
\]

These observations are now extended to an arbitrary value of the base
$b$.
\begin{prop}
\label{prop:The-structure-of}The structure of the triangle built
from the coefficients $\binom{n}{k}_{b}$ satisfies \end{prop}
\begin{enumerate}
\item $T_{1}^{\left(b\right)}$ is made of the first $b$ rows of the usual
Pascal triangle;
\item for $m>1,$ $T_{m}^{\left(b\right)}$ is obtained from $T_{m-1}^{\left(b\right)}$
by the associative operation
\[
T_{m}^{\left(b\right)}=T_{m-1}^{\left(b\right)}\otimes T_{1}^{\left(b\right)},
\]
defined as 
\[
T_{m}^{\left(b\right)}=\begin{array}{ccccccc}
 &  &  & \binom{0}{0}T_{m-1}^{\left(b\right)}\\
\\
 &  & \binom{1}{0}T_{m-1}^{\left(b\right)} &  & \binom{1}{1}T_{m-1}^{\left(b\right)}\\
 & \iddots &  &  &  & \ddots\\
\binom{b-1}{0}T_{m-1}^{\left(b\right)} &  & \ldots & \binom{b-1}{j}T_{m-1}^{\left(b\right)} & \ldots &  & \binom{b-1}{b-1}T_{m-1}^{\left(b\right)}
\end{array}
\]
so that 
\[
T_{m}^{\left(b\right)}=\left[T_{1}^{\left(b\right)}\right]^{\otimes m};
\]

\item $T_{m}^{\left(b\right)}$ has $b^{m}$ rows since 
\[
\#rows\left(T_{m}^{\left(b\right)}\right)=b\times\#rows\left(T_{m-1}^{\left(b\right)}\right);
\]

\item the bottom row of $T_{m}^{b}$ has $b^{m}$ entries.\end{enumerate}
\begin{proof}
These properties are a direct consequence of (\ref{eq:binomial coefficients}).
Since properties (3) and (4) are elementary, only properties (1) and
(2) need to be verified and here we use induction.

(i) For $m=1$, $T_{1}^{\left(b\right)}$ contains the first $b$
rows of the triangle, made of coefficients $\binom{n}{k}_{b}$where
\[
0\le n\le b-1,
\]
so that the $b-$ary expression of $n$ consists of a single digit:$n=n_{0}$.
In this case, $\binom{n}{k}_{b}$ coincides with $\binom{n}{k}$ making
the first $b$ rows coincide with those of the Pascal triangle (see
Remark \ref{rem:remark 3.2})

(ii) Consider $T_{m}^{\left(b\right)}$ by assuming that properties
(1) and (2) hold for $T_{1}^{\left(b\right)},\dots T_{m-1}^{\left(b\right)}$.
Then, the $b^{m}$ elements are exactly the case
\[
0\le n\le b+\cdots+b^{m}-1,
\]
which implies that $n$ could have at most $m$ digits, namely
\[
n=n_{m-1}b^{m-1}+\cdots+n_{0}.
\]
Thus, 
\[
{n \choose k}_{b}={n_{m-1} \choose k_{m-1}}\underset{\text{Copy of }T_{m-1}^{\left(b\right)}}{\underbrace{\overset{m-2}{\underset{l=0}{\prod}}{n_{l} \choose k_{l}}}}.
\]
Since $0\le n_{m-1}\le b-1$, ${n_{m-1} \choose k_{m-1}}$ gives a
copy of $T_{1}^{\left(b\right)}$ while the rest of the product gives
a copy of $T_{m-1}^{\left(b\right)}$. Thus, by induction
\[
T_{m}^{\left(b\right)}=\left[T_{1}^{\left(b\right)}\right]^{\otimes m}.
\]

\end{proof}

\thanks{The work of the second author was partially supported by the iCODE
Institute, Research Project of the Idex Paris-Saclay. The authors
wish to thank V.H. Moll for his support and discussions.}

\end{document}